\documentclass[12pt]{amsart}
\usepackage{amsmath,amsthm,amssymb,amsfonts}
\usepackage{latexsym}
\usepackage{enumerate}
\usepackage{verbatim}
\usepackage{multicol}
\usepackage{graphicx}
\usepackage{epsfig}
\usepackage{color}
\usepackage{bezier}
\usepackage{curves}
\usepackage{fullpage}
\usepackage{MnSymbol}
\usepackage{trfsigns}
\usepackage{cases}

\newtheorem{theorem}{Theorem}

\newtheorem{problem}{Problem}
\newtheorem{remark}{Remark}
\newtheorem{lemma}[theorem]{Lemma}

\newtheorem{corollary}[theorem]{Corollary}

\newtheorem{definition}{Definition}

\newcommand{\NN}{{\mathbb N}}
\newcommand{\ZZ}{{\mathbb Z}}

\DeclareMathOperator{\ainc}{Inc}

\DeclareMathOperator{\comp}{{Comp}}

\DeclareMathOperator{\prim}{Prim}

\DeclareMathOperator{\age}{Age}
\DeclareMathOperator{\inc}{Inc}

\title[Minimal prime ages, words and permutation graphs]{Minimal prime ages, words and permutation graphs\\ Extended abstract}

\author[D.Oudrar] {Djamila Oudrar}
\address{Faculty of Mathematics, USTHB, Algiers, Algeria}
\email {dabchiche@usthb.dz}

\author [M.Pouzet]{Maurice Pouzet}
\address{Univ. Lyon, Universit\'e Claude-Bernard Lyon1, CNRS UMR 5208, Institut Camille Jordan, 43, Bd. du 11 Novembre 1918, 69622
Villeurbanne, France et Department of Mathematics and Statistics, University of Calgary, Calgary, Alberta, Canada}
\email{pouzet@univ-lyon1.fr }

\author[I.Zaguia]{Imed Zaguia*}\thanks{*Corresponding author. Supported by Canadian Defence Academy Research Program, NSERC and LABEX MILYON (ANR-10-LABX-0070) of Universit\'e de Lyon within the program ''Investissements d'Avenir (ANR-11-IDEX-0007'' operated by the French National Research Agency (ANR)}
\address{Department of Mathematics \& Computer Science, Royal Military College of Canada,
P.O.Box 17000, Station Forces, Kingston, Ontario, Canada K7K 7B4}
\email{zaguia@rmc.ca}

\date{\today}

\begin{document}


\begin{abstract} This paper is a contribution to the study of hereditary classes of finite graphs.
We  classify these classes according to the number of prime structures they contain.
We consider such classes that are \emph{minimal prime}: classes that contain infinitely many primes but every proper hereditary subclass contains only   finitely many primes. We give a complete characterization of such classes. In fact, each one of these classes is a well quasi ordered age and there are uncountably many of them. Eleven  of these ages remain well quasi ordered when labels in a well quasi ordering are added. Among the remaining ones, countably many remain well quasi ordered when one  label is added.

Except for six examples, members of these ages we characterize are permutation graphs. In fact, every age which is not among the eleven ones is the age of a graph associated to a uniformly recurrent $0$-$1$ word  on  the integers.

A characterization of minimal prime classes of posets and bichains is also provided.
\end{abstract}

\maketitle

\section{Introduction}
This paper is a contribution to the study of hereditary classes of finite graphs. We  classify hereditary classes according to the number of prime structures they contain. We consider first hereditary classes of graphs that contain only finitely many prime members. Then, we consider those hereditary classes which contain  infinitely many prime members, and we show that there are  minimal ones with respect to set inclusion. We obtain some general results that we are able to refine in some special cases like  graphs, ordered sets,  and bichains.

This paper is mostly about  graphs and posets. We will also  consider binary relational structures, that is ordered pairs $\mathcal R:=(V,(\rho_{i})_{i\in I})$ where each $\rho_i$ is a binary relation or a unary relation on $V$. The sequence $s:= (n_i)_{i\in I}$ of arity  $n_i$ of $\rho_i$ is the \emph{signature} of $\mathcal R$. We denote by $\Omega_s$ the collection of finite structures of signature $s$. In the sequel we will suppose the signature finite, i.e. $I$ finite. We present the main notions in terms of graphs.
Unless otherwise stated, the graphs we consider are undirected, simple and have no loops. That is, a {\it graph} is a
pair $G:=(V, E)$, where $E$ is a subset of $[V]^2$, the set of $2$-element subsets of $V$. Elements of $V$ are the {\it vertices} of
$G$ and elements of $ E$ its {\it edges}. The {\it complement} of $G$ is the graph $\overline{G}$ whose vertex set is $V$ and edge set
${\overline { E}}:=[V]^2\setminus  E$. If $A$ is a subset of $V$, the pair $G_{\restriction A}:=(A,  E\cap [A]^2)$ is the \emph{graph
induced by $G$ on $A$}. We compare graphs with the embeddability relation. A graph $G$ is \emph{embeddable} in a graph $G'$ and we set $G\leq G'$,  if $G$  is isomorphic to an induced subgraph of $G'$. This defines a quasi order. We recall that  a class $\mathcal C$ of graphs is \emph{hereditary}   if it contains every graph  $G$ which embeds in some member of  $\mathcal C$. Such a class  is an initial segment of the class of graphs quasi ordered by embeddability.
If $G$ is a graph, then the \emph{age} of $G$ is the collection  $\age(G)$ of finite graphs $H$, considered up to isomorphy, which embed in  $G$.  A characterization of ages (also valid for classes of relational structures with a finite signature) was given by Fra\"{\i}ss\'e (see chapter 10 of  \cite{fraissetr}). Namely, a class $\mathcal C$ of finite graphs is the age of some  graph if and only if $\mathcal C$  is an \emph{ideal} of the class of finite graphs, that is  a  nonempty,  hereditary  and  \emph{up-directed} class (a class in which any pair of members of $\mathcal C$ are embeddable in some element of $\mathcal C$). We   recall that an ordered $P:= (V, \leq)$  is \emph{well quasi ordered}  (w.q.o) if every sequence $x_0, \ldots, x_n, \ldots$ contains an increasing subsequence $x_{n_0}\leq  \ldots, \leq x_{n_k}, \ldots$ with respect to embeddability. If the class $P$ does  not contain infinite descending chains, this amounts to the nonexistence of  infinite antichains. A
 class $\mathcal C$ of graph or more generally of relational structures  is \emph{hereditary well-quasi-ordered} if the class of members of $\mathcal C$ labelled by any  w.q.o. is w.qo. with respect to the embeddability quasi order. A \emph{bound} of a hereditary class $\mathcal C$ of finite structures is any minimal structure not  in $\mathcal C$. We recall that a hereditary class $\mathcal C$ of finite structures which is hereditary w.q.o. has finitely many bounds \cite{pouzet72}.

\section{Minimal prime hereditary classes}\label{section:min-hered-class}
%
We start  with the  notion of a module.

\begin{definition}\label{def:module}
Let  $\mathcal{R}:=(V,(\rho_i)_{i\in I})$ be a binary relational structure. A  \emph{module} of $\mathcal{R}$ is any subset $A$ of $V$ such that  $$(x\rho_i a \Leftrightarrow x\rho_i a')  \; \text{and} \; (a\rho_i x \Leftrightarrow a'\rho_i x) \; \text{for all} \; a,a'\in A \;\text{and}\;
x\notin A \; \text{and} \; i\in I.$$
\end{definition}

The empty set, the singletons in $V$ and the whole set $V$ are modules and are called \textit{trivial}. (sometimes in the literature, modules  are called \emph{interval},
 \emph{autonomous} or \emph{partitive sets}). If $\mathcal{R}$ has no nontrivial module, it is called \emph{prime} or \textit{indecomposable}.

\noindent For example, if $\mathcal{R}:=(V,\leq)$ is a chain, its modules  are the ordinary intervals of the chain. If $\mathcal{R}:=(V,\leq,\leq')$ is a bichain then $A$ is a  module of $\mathcal{R}$ if and only if $A$ is an  interval of $(V,\leq)$ and $(V,\leq')$.

The notion of module  goes back to Fra\"{\i}ss\'e  \cite{fraisse3} and Gallai  \cite{gallai}, see also \cite{fraisse84}. A fundamental decomposition result of a binary structure into modules was obtained by Gallai \cite{gallai} for finite binary relations (see \cite{ehren} for further extensions).
We recall the  compactness result of Ille \cite{ille}.

\begin{theorem}\label{ille-theorem}
A binary structure $\mathcal {R}$  is prime if and only if every finite subset $F$ of its domain extend to a finite sets $F'$ such that  $\mathcal {R}_{\restriction F'}$ is prime.
\end{theorem}

We consider the class $\prim_{s}:=\prim (\Omega_s)$ of finite  binary structures of signature $s$ which are prime. We set $\prim (\mathcal C):= \prim_s\cap \, \mathcal C$ for every $\mathcal C\subseteq \Omega_s$.

We say that a subclass $\mathcal{D}$ of $\prim_{s}$ is \emph{hereditary} if it contains every member of $\prim_{s}$  which can be embedded into some member of $\mathcal{D}$.

\subsection{Hereditary classes containing finitely many prime structures}

The following result (see Proposition 5.2  of \cite{oudrar-pouzet2016}) improves  a result of  \cite{albert-atkinson} for  hereditary classes of finite permutations.

\begin{theorem}\label{thm:finite-prime}Let $\mathcal {C}$ be a hereditary class of finite binary structures containing only finitely many prime structures. Then $\mathcal {C}$ is hereditarily w.q.o. In particular, $\mathcal {C}$ has finitely many bounds.
\end{theorem}

The following result, due independently to C. Delhomm\'e \cite{delhomme1} and McKay
 \cite{mckay1} extends Thomass\'e's result on the w.q.o.  character of series-parallel posets \cite{thomasse},  which extends the  famous Laver's theorem \cite{laver} on the  w.q.o.  character of the class of countable chains.
\begin{theorem}
Let $\mathcal C$ be a hereditary classes of $\Omega_{s}$. If $\prim (\mathcal C)$ is finite, then  the collection  of countable $R$ such that $Age(R)\subseteq \mathcal C$ is well-quasi-ordered  by embeddability.
\end{theorem}

\subsection{Hereditary classes containing infinitely many prime structures}

In this subsection, we report some results included in \cite{oudrar}. We consider hereditary classes containing infinitely many prime structures. We show that each  such a class contains one which is minimal with respect to inclusion.

\begin{definition} A hereditary class $\mathcal C$ of $\Omega_{\mu}$ is \emph{minimal prime} if it contains infinitely many prime structures, while every  proper  hereditary subclass contains only finitely  many prime structures.
\end{definition}
This notion appears in the thesis of the first author \cite{oudrar} (see Theorem 5.12, p. 92, and Theorem 5.15, p. 94 of  \cite{oudrar}).

Due to their definition, minimal prime ages ordered by inclusion form an antichain with respect to set inclusion.

An ordered set $P$ is \emph{J\'{o}nsson} if $P$ is infinite and the cardinality of every proper initial segment of $P$ is strictly less than the cardinality of $P$ \cite{kearnes}. We say that $P$ is \emph{minimal} if it is infinite and every proper initial segment of $P$ is finite. This amounts to say that $P$ is a countable J\'{o}nsson poset. These posets appear quite naturally in symbolic dynamic. In fact,  an infinite word $u$ on some finite alphabet $A$ is uniformly recurrent (\cite{All-Sha}, \cite{lothaire}) if and only if the set $Fac(u)$ of its finite factors is minimal once it is ordered with the factor order.

 We list below some equivalent properties  see Proposition 4.1 of \cite{pouzet-sauer}, or Proposition 3.1 of \cite{assous-pouzet}.

\begin{theorem} \label{minimalposet} Let $P$ be an
infinite  poset. Then, the following properties are  equivalent:
\begin{enumerate}[{(i)}]
\item Every proper initial segment of $P$ is finite.
\item $P$ is w.q.o.  and all ideals distinct from $P$ are principal;
\item $P$ has no infinite antichain and all ideals distinct from $P$ are finite;
\item $P$ is level-finite, of height  $\omega$, and
 for each $n<\omega$ there is  $m<\omega$ such that each element of
height at most $n$ is below every element of height at least $m$.
 \end{enumerate}
\end {theorem}

We have immediately (cf. Th\'eor\`eme 5.14 p.93 of \cite{oudrar}).

\begin{theorem} \label{thm:minimalprime}A hereditary class $\mathcal C$ of $\Omega_s$ is minimal prime if and only if $\prim  (\mathcal C)$
is a J\'onsson poset which is cofinal in $\mathcal C$.
\end{theorem}

\begin{proof} Let $\mathcal  C$ be a minimal prime class. By definition, $\prim  (\mathcal C)$ is infinite.
Let $\mathcal I$ be a proper hereditary subclass of $\prim  (\mathcal C)$. The initial segment $\downarrow \mathcal I$ in $\Omega_s$ is a proper subclass of $\mathcal C$. Hence $\mathcal I$ is finite. Thus $\prim (\mathcal C)$ is J\'onsson. Let $\mathcal C':= \downarrow \prim  (\mathcal C)$. If $\mathcal C'\not =\mathcal C$ then since $\mathcal C$ is minimal prime, $\prim  (\mathcal C')= \prim  (\mathcal C)$ is finite, which is impossible. This proves that the implication holds

 Conversely,  suppose that $\prim  (\mathcal C)$ is a J\'onsson poset which is cofinal in $\mathcal C$. Then
$\mathcal C$ est infinite. If $\mathcal C$ is not  minimal prime there is a	proper hereditary subclass $\mathcal C'$  of $\mathcal C$ such that $\prim  (\mathcal C')$ is infinite. Since $\prim (\mathcal C)$  is J\'onsson, $\prim  (\mathcal C')= \prim  (\mathcal C)$. Since $\mathcal C'= \downarrow \prim  (\mathcal C')$ and $\prim  (\mathcal C)$ is cofinal in $\mathcal C$, this yields $\mathcal C'= \mathcal C$, a contradiction.
\end{proof}

We have:

\begin{theorem}\label{minimal}
Every   hereditary subclass of  finite graphs, and more generally of  finite relational structures (with a given finite signature), which contains  infinitely many prime structures   contains a minimal prime hereditary subclass.
\end{theorem}

With Theorem \ref{ille-theorem}, one gets:
\begin{corollary} The age of any infinite prime structure contains a minimal prime age.
\end{corollary}


With  Theorems \ref{minimalposet} and \ref{thm:finite-prime} we get:

\begin{theorem} \label{thm:main1} Every  minimal prime hereditary class is the age of some prime structure; furthermore this age is well-quasi-ordered.
\end{theorem}

A non-trivial improvement  of Theorem \ref{thm:main1} is based on the notion of kernel (a notion introduced in \cite{pouzettr} and studied it in several papers \cite{pouzet-minimale} \cite{pouzet.81}, \cite{pouzet-impartible} Lemme IV-3.1 p. 37 and \cite{pouzet-sobranisa}).

The \emph{ kernel } of  a relational structure $\mathcal{R}$ with domain $E$ is
the subset $K(R)$ of $x\in
E$ such that  $Age(\mathcal{R}\restriction_{E\setminus \{x\}})\not = Age (\mathcal {R})$.
 As it is easy to see (cf \cite{pouzet-minimale}\cite{pouzet-sobranisa}), the kernel of a relational structure $R$ is empty if and only if
for every finite subset $F$ of $E$ there is a disjoint subset $F'$ such that the
restrictions $R\restriction_ F$ and $R\restriction_{F'}$ are isomorphic. Hence,
relational structures with empty kernel are those for which their age has the \emph{disjoint embedding property}, meaning that two arbitrary members of the age can be embedded into a third in such a way that their domain are disjoint.


\begin{theorem} \label{thm:main2}Let $\mathcal C$ be a minimal prime class. If the kernel of some structure $\mathcal{R}$ such that $Age (\mathcal{R})=\mathcal C$ is non-empty,  then $\mathcal C$ is hereditarily wqo and, in particular,  this kernel is finite.

\end{theorem}

\begin{corollary} There are at most countably many minimal prime classes $\mathcal C$  such that $\mathcal C =Age (\mathcal{R})$ and $Ker(\mathcal{R})\not =\emptyset$.
\end{corollary}
\begin{problem}
Is it true that $\vert Ker(\mathcal{R})\vert \leq 2$?
\end{problem}
As we will see, the  answer is positive if one considers minimal prime classes of graphs. In this case, there are only four leading to nonempty kernel.

\subsection{A Proof of Theorem $\ref{minimal}$.} We will need the following lemma which is a special case of Theorem 4.6 of \cite{assous-pouzet}.

\begin{lemma}\label{lem:levelfinite-prims}$\prim_{s}$ is level finite.
\end{lemma}
\begin{proof}Suppose for a contradiction that there exists an integer $n\geq 0$ such that the level ${\prim_s}(n)$ of $\prim_s$ is infinite and choose $n$ smallest with this property. Define
$$\mathcal{C}:=\{R\in \Omega_{s}: R< S \text{ for some } S\in {\prim}_{s}(n)\}.$$
Then $\mathcal{C}$ is a hereditary class of $\Omega_{s}$ containing only finitely many prime structures. It follows from Theorem \ref{thm:finite-prime} that $\mathcal{Cer}$ is hereditary well-quasi-ordered and hence has finitely many bounds. This is not possible since the elements of ${\prim_s}(n)$ are  bounds of $\mathcal{C}$.
\end{proof}

\subsubsection{Another proof of Lemma \ref{lem:levelfinite-prims}}
We prove the finiteness of the levels of $\prim_{s}$ via the properties of critical primality. A binary structure $R:=(V,(\rho_i)_{i\in I})$ is \textit{critically prime} if it is prime and   $R_{V\setminus \{x\}}$ is not prime for every $x\in V$.  Note that $\vert R\vert$ has at least four elements.
This notion of critical primality was introduced by Schmerl and Trotter \cite{S-T}. Among results given in their paper, we have the following theorem (this is Theorem 5.9, page 204):

\begin{theorem}\label{theo:indec}
 Let $R :=(V,(\rho_i)_{i\in I})$ be a prime binary structure of order $n\geqslant 7$. Then there are distinct $c, d\in E$ such that $V\setminus \{c,d\}$ is prime.
\end{theorem}

%
In their paper, Schmerl and Trotter give examples of critically prime structures within the class of graphs, posets, tournaments, oriented graphs and binary relational structures. The set of critical prime structures within each of these classes  is a finite union of chains.

Decompose $Prim_s$  into levels;  in level $i$, with $i\leq 2$,
are the structures of order zero, one or two.

For structures $R$ in $Prime_s$ of order at least $2$, we have the following  relationship between the height $h(R)$ in $Prime_s$   and its order, $\vert R \vert$ (which is the height of $\mathcal R$ in $\Omega_s$).

\begin{equation} \label{ineq:1}
    h(R) \leqslant \vert R \vert  \leqslant 2(h(R)-1).
      \end{equation}

The first inequality is obvious. For the second, we use induction on $n:= h(R)\geq 2$. The basis step $n=2$ is trivially true. Suppose $n>2$. Let $S$ be prime such that $S$ embeds in $R$ with $h(S)=n-1$. From the induction hypothesis, $\vert S\vert \leq 2(h(S)-1)= 2(n-2)$. According to Theorem \ref {theo:indec}, $\vert R \vert -2 \leq \vert S \vert $. Hence $\vert R\vert -2 \leq 2(n-2)$. Therefore $\vert R\vert \leq 2(n-1)$.

Lemma \ref{lem:levelfinite-prims} follows from the second inequality in (\ref{ineq:1}) since there are only finitely many structures of a given order.

\begin{lemma} \label{lem:contains minimal}Every infinite well-founded poset $P$ which is level finite contains an  initial segment which is J\'onsson.
\end{lemma}
\begin{proof}
We  apply Zorn's Lemma to the set  $\mathcal J$ of infinite initial segments of $P$ included in the first $\omega$-levels. For that, we prove that $\mathcal J$ is closed under intersections of nonempty chains. Indeed,  let $\mathcal C$ be a non-empty chain (with respect to set inclusion) of  members of $\mathcal J$. Set $J:= \cap \mathcal C$. Let $n<\omega$, let $P_n$ be the $n$-th level of $P$ and $\mathcal {C}_n:=  \{C\cap P_n: C \in \mathcal C\}$. The members of $\mathcal {C}_n$ are finite, nonempty and totally ordered by inclusion. Hence,  $J_n:=\cap \mathcal {C}_n$ is non-empty. Since $J=\cup \{ J_n: n\in \NN\}$, $J\in \mathcal J$.
\end{proof}

The proof of Theorem $\ref{minimal}$ goes as follows. Let $\mathcal C$ be a hereditary class of $\Omega_s$ such that $J:= \prim_s(\mathcal C)$ is infinite. Since $\prim_{s}$ is level finite,   Lemma \ref {lem:contains minimal} ensures that $J$ contains an initial segment $D$ which is J\'onsson. According to Theorem \ref{thm:minimalprime}, $\downarrow D$ is minimal prime. This completes the proof.\hfill $\Box$

\subsection{A Proof of Theorem  $\ref {thm:main1}$.}  Let $\mathcal{C}$ be a minimal prime hereditary class. We first prove that it is the age of a prime structure. It follows from Theorem \ref{thm:minimalprime} that $\mathcal{C}=\downarrow D$ where $D$ is J\'{o}nsson. Since $D$ is J\'{o}nsson, it is up-directed. Thus $\mathcal{C}$ is an age. Since $D$ is up-directed and countable,  it contains a cofinal sequence $R_0\leq R_1\leq \ldots <R_n\leq \ldots$. We may define the limit $R$ of these $R_n$. Since the $R_n$'s are prime,  $R$ is prime and $\age(R)=\mathcal{C}$.

Next we prove that $\mathcal{C}$ is w.q.o.  Since $D$ is J\'{o}nsson,  it is w.q.o. . To prove that $\mathcal{C}$ is w.q.o., let $R\in \mathcal{C}$ and consider $\mathcal{C}\setminus (\uparrow \{R\}$). In order to prove that $\mathcal{C}$ is w.q.o.  it is enough to prove that $\mathcal{C}\setminus (\uparrow \{R\}$)  is w.q.o.  by embeddability. Indeed, an antichain that contains $R$ must be in $\mathcal{C}\setminus (\uparrow \{R\}$). Now to prove that $\mathcal{C}\setminus (\uparrow \{R\}$) is w.q.o.  we note that since $\mathcal{C}\setminus (\uparrow \{R\}$) is a proper hereditary class in $\mathcal{C}$, hence it contains only finitely many primes. It follows from Theorem \ref{thm:finite-prime}  that $\mathcal{C}\setminus (\uparrow \{R\})$  is w.q.o. . \hfill $\Box$

\section{Minimal prime ages of graphs}
Our characterization of minimal prime ages is based on a previous characterization of unavoidable prime graphs in large finite prime graphs \cite{chudnovsky}, and our study of graphs associated to $0$-$1$ sequences.

Citing Chudnovsky and al \cite{chudnovsky} :

\begin{theorem}\label{thm:chudnovsky}
For all $n$, there exists $N$ such that every prime graph with at least $N$ vertices contains one of the following graphs or their complements as an induced subgraph:
\begin{enumerate}
\item the graph obtained from $K_{1,n}$ by subdividing every edge once,
\item the line graph of $K_{2,n}$,
\item the line graph of the graph in $(1)$,
 \item the half-graph of height $n$,
 \item a prime graph induced by a chain of length $n$,
 \item  two particular graphs obtained from the half-graph of height $n$ by making one side a clique and adding one vertex.
\end{enumerate}
\end{theorem}

It turns out that we can represent a chain by a word on the alphabet $\{0,1\}$.

In this text, a $0$-$1$ sequence is a map $\mu$  from an interval $I$ of the set $\ZZ$ of integers into $\{0,1\}$. The restriction of $\mu$ to an interval of $I$ is a \emph{factor} of $\mu$. For $i\in I$ we will sometimes denote $\mu(i)$ by $\mu_i$.  If $I$ is finite, with $n$ elements, we may view $\mu$  as a word  $\mu:=  u_0\ldots   u_{n-1}$. If the sequence is infinite we may view it as a $0$-$1$ sequence over $\NN$, over $\NN^{*}:= \{\ldots, -n, \ldots, -2, -1, 0\}$, or over $\ZZ$.

\begin{definition}To $\mu$ we associate the graph $G_{\mu}$ whose vertex set $V(G_\mu)$ is $\{-1,  0, \ldots,  n-1\}$  if  the domain of $\mu$ is $\{0, \ldots,  n-1\}$, $\{-1\} \cup \NN$ if the domain of $\mu$  is $\NN$, and $\NN^{*}$  or $\ZZ$ if the domain of $\mu$ is $\NN^{*}$  or $\ZZ$. For two vertices $i,j$ with $i<j$ we let $\{i,j\}$ be an edge of $G_{\mu}$ if and only if
\begin{align*}
  \mu_j=1 & \mbox{ and } j=i+1,\mbox{or} \\
  \mu_j=0 &\mbox{ and } j\neq i+1.
\end{align*}
\end{definition}


For instance, if $\mu$ is the word defined on $\NN$ by setting $\mu_i=1$ for all $i\in \NN$, then  $G_\mu$ is the infinite one way path on $\{-1\}\cup \NN$. Note that if $\mu'$ is the word defined on $\NN$ by setting $\mu'_i=1$ for all $i\in \NN \setminus \{1\}$ and $\mu'_1=0$, then  $G_{\mu'}$ is also the infinite one way path. In particular the graphs $G_{\mu}$ and $G_{\mu'}$ have the same age but $\mu$ and $\mu'$ do not have the same sets of finite factors.

\begin{figure}[h]
\begin{center}
\leavevmode \epsfxsize=5in \epsfbox{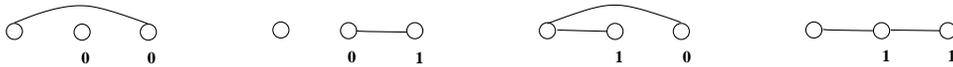}
\end{center}
\caption{$0$-$1$ words of length two and their corresponding graphs.} \label{fig:gmu-three}
\end{figure}

\begin{figure}[h]
\begin{center}
\leavevmode \epsfxsize=4in \epsfbox{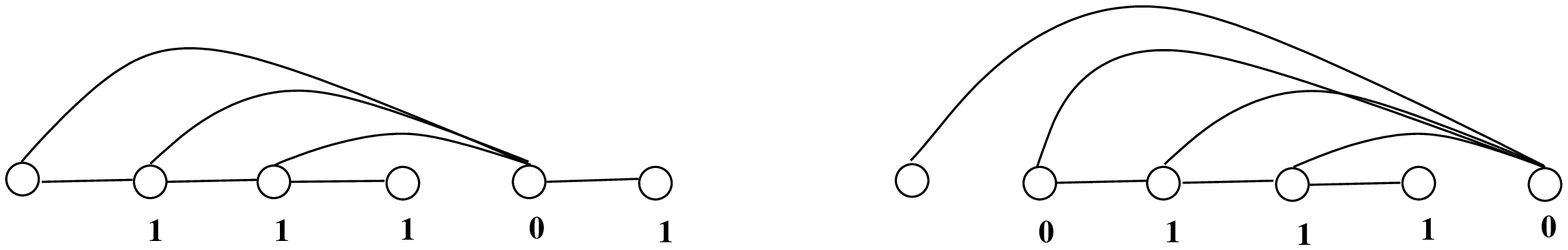}
\end{center}
\caption{Two distinct $0$-$1$ sequences with isomorphic corresponding graphs.} \label{fig:gmu-same-graphs}
\end{figure}

This correspondence between $0$-$1$ words and graphs was first considered in \cite{sobranithesis}, \cite{sobranietat}; see also  \cite{zverovich} and \cite{chudnovsky}.
\begin{remark}\label{lem:comp}
If $I$ is an interval of $\NN$ and $\mu:=(\mu_{i})_{i\in I}$ is a $0$-$1$ sequence, then $\overline{G_\mu}=G_{\overline{\mu}}$, where $\overline{\mu}:=(\overline{\mu}_{i})_{i\in I}$ is the $0$-$1$ sequence defined by $\overline{\mu}(i):=\mu(i)\dot +1$ and $\dot +$ is the addition modulo $2$.
\end{remark}

\begin{figure}[h]
\begin{center}
\leavevmode \epsfxsize=3in \epsfbox{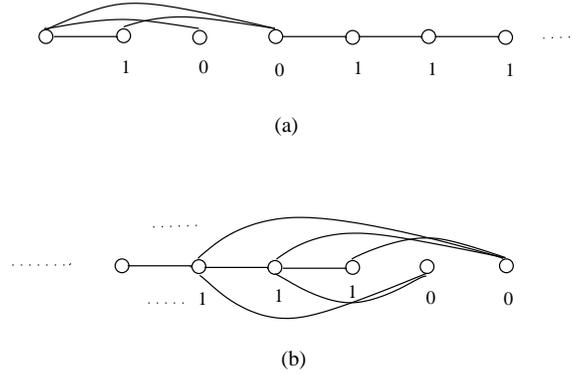}
\end{center}
\caption{$0$-$1$ graphs nonrealizable by a sequence on $\ZZ$.} \label{fig:gmu-nonrealizable}
\end{figure}

\begin{remark} Given a $0$-$1$ graph defined on $\NN\cup \{-1\}$ or on $\NN^*$ there does not exist necessarily a $0$-$1$ graph on $\ZZ$ with the same age.\\
\emph{Indeed,
$(a)$ Let $\mu:=100111\ldots$ be an infinite word on $\NN$ (the corresponding graph is depicted in (a) of  Figure $\ref{fig:gmu-nonrealizable}$). There does not exist a word $\mu'$ on $\NN^*$ or $\ZZ$ such that $\age(G_\mu)=\age(G_{\mu'})$. \\
$(b)$ Let $\nu:=\ldots 11100$ be an infinite word on $\NN^*$ (the corresponding graph is depicted in (b) of  Figure $\ref{fig:gmu-nonrealizable}$). There does not exist a word $\nu'$ on $\NN$ or $\ZZ$ such that $\age(G_\nu)=\age(G_{\nu'})$.\\
\emph{Proof of (a):} Every vertex of the graph $G_\mu$ has finite degree. Suppose for a contradiction that there exists a word $\mu'$ on $\NN^*$ or $\ZZ$ such that $\age(G_\mu)=\age(G_{\mu'})$. Then there exists $i \in \ZZ$ such that $\mu'(i)=0$ because otherwise $G_{\mu'}$ would be a path and hence $\age(G_\mu)\neq \age(G_{\mu'})$. But then the vertex $i$ of $G_{\mu'}$ would have infinite degree which is impossible since every vertex of the graph $G_\mu$ has finite degree. $\hfill$ $\blacksquare$\\
\emph{Proof of (b):} The graph $G_\nu$ has two vertices of infinite degree. Suppose for a contradiction that there exists a word $\nu'$ on $\NN$ or $\ZZ$ such that $\age(G_\nu)=\age(G_{\nu'})$. Then $\nu'$ must take the value $0$ on an infinite subset of $I$ of $\ZZ$ because otherwise every vertex of $G_{\nu'}$ would have finite degree which is impossible since $\age(G_\nu)=\age(G_{\nu'})$. Let $I'\subseteq I$ be an infinite set of nonconsecutive integers. Then $G_{\nu'}$ induces an infinite clique on $I'$. This is not possible since the only cliques of $G_\nu$ have cardinality $3$}.   $\hfill$ $\blacksquare$
\end{remark}

\begin{remark} Given a word $\nu$ we associate the graph $G{\nu}$ whose vertex set $V(G^\nu)$ is $\{-n+1, \ldots,  0,1\}$  if  the domain of $\nu$ is $\{-n+1, \ldots,  0\}$, $\NN$ or $\ZZ$ if the domain of $\nu$  is $\NN$ or $\ZZ$ respectively, and $\NN^{*}\cup \{1\}$   if the domain of $\nu$ is $\NN^{*}$. For two vertices $i,j$ with $i<j$ we let $\{i,j\}$ be an edge of $G^{\nu}$ if and only if
\begin{align*}
  \nu_i=1 & \mbox{ and } j=i+1,\mbox{or} \\
  \nu_i=0 &\mbox{ and } j\neq i+1.
\end{align*}
If $\nu$ is of domain $\{0, \ldots,  n-1\}$, $\NN$, $\NN^*$ or $\ZZ$ define $\nu^*$ to be the sequence of domain is $\{-n+1, \ldots,  0\}$, $\NN^*$, $\NN$ or $\ZZ$ respectively by setting $\nu^*(i):=\nu(-i)$. Then ${G}^{\nu^*}$ and $G_\nu$ are isomorphic.
\end{remark}

A graph $G:= (V, E)$ is a \emph{permutation graph} if there is a total order $\leq $ on $V$ and a permutation $\sigma$ of $V$ such that the edges of $G$ are the pairs  $\{x, y\}\in [V]^2$ which are reversed by $\sigma$. A graph is  the  comparability graph of a two-dimensional poset if and only if it is also the  incomparability graph of a two-dimensional poset \cite{dushnik-miller}. If the graph is finite,  this amounts to the fact that this is a permutation graph.

During the last fifteen years, several  studies have been devoted to permutation graphs and some variants, in relation with the Stanley-Wilf conjecture and its solution by Marcus and Tard\"os \cite{marcus-tardos}. An emphasis was put on hereditary classes of finite permutation graphs and a classification of these classes, notably in terms of their profile. The role of the notions of primality and of well-quasi-order has been particularly investigated,  see \cite {klazar, vatter}.

\begin{theorem}\label{thm:permutation-graph}For every $0$-$1$ word $\mu$ the age  $\age(G_\mu)$ consists of permutation graphs.
\end{theorem}

Theorem \ref{thm:permutation-graph} follows from the  Compactness Theorem of First Order Logic and the following Lemma.

Let $P: (V, \leq)$ be a poset. An element $x\in V$ is \emph{extremal} if it is maximal or minimal.

\begin{lemma}\label{lem:one-extension}Let $w:=w_0\ldots w_{n-1}$ be a finite word with $n\geq 2$ and $w':=w_0\ldots w_{n-2}$. Then every realizer $(L_{w'}, M_{w'})$ of a transitive orientation of $G_{w'}$ on $\{-1, 0, \dots, n-2\}$ (if any) such that $n-2$ is extremal in $L_{w'}$ or in $M_{w'}$ extends to a realizer $(L_{w}, M_{w})$ of a transitive orientation of $G_{w}$ on $\{-1, 0, \dots, n-1\}$   such that $n-1$ is extremal in $L_{w}$ or in $M_{w}$.
\end{lemma}
\begin{proof}
Let    $(L_{w'}, M_{w'})$ be a realizer of  a transitive orientation $P_{w'}$ of $G_{w'}$ on $\{-1, 0, \dots, n-2\}$ such that $n-2$ is extremal in $L_{w'}$ or in $M_{w'}$.
 We may assume without loss of generality that $n-2$ is maximal in $L_{w'}$ or in  $M_{w'}$. Otherwise, consider $P^*_{w'}$ and the pair $(L^*_{w'}, M^*_{w'})$. Note that $P^*_{w'}$ is a transitive orientation of $G_{w'}$,  the pair  $(L^*_{w'}, M^*_{w'})$  is a realizer of $P^*_{w'}$,     and $n-2$ is maximal in  $L^*_{w'}$ or in $M^*_{w'}$ (this is because  $n-2$ is minimal in $L_{w'}$ or in $M_{w'}$). We then extend   $(L^*_{w'}, M^*_{w'})$  to a realizer of  $P^*_{w'}$  with the desired property. The dual of this realizer is a realizer of $P_{w}$ with the required property. We may also suppose that $n-2$ is maximal in $L_{w'}$, because  otherwise, we interchange the roles of  $L_{w'}$ and $M_{w'}$. \\
$\bullet$ If $w_{n-1}=1$, then $\{n-2,n-1\}$ is the unique edge of $G_{w}$ containing $n-1$. Clearly $P_w:=P_{w'}\cup \{(n-1,n-2)\}$ is a transitive orientation of $G_{w}$. Let $L_w$ be the total order obtained from $L_{w'}$ so that $n-1$ appears immediately before $\max(L_{w'})=n-2$ and larger than all other elements and let $M_w$ be the total order obtained from $M_{w'}$ by letting $n-1$ smaller than all elements of $M_{w'}$. Clearly, $(L_w, M_w)$ is a realizer of  $P_{w}$ and by construction $n-1$ is minimal in $P_{w}$ and in $M_w$.\\
$\bullet$ Else if $w_{n-1}=0$, then $\{n-2,n-1\}$ is the unique non edge of $G_{w}$ containing $n-1$. Since $n-2$ is maximal in $P_{w'}$ we infer that $P_w:=P_{w'}\cup \{(x,n-1)) : x\in \{-1,0,\ldots,n-3\}\}$ is a transitive orientation of $G_{w}$ in which $n-1$ and $n-2$ are incomparable. Let $L_w$ be the total order obtained from $L_{w'}$ so that $n-1$ appears immediately before $\max(L_{w'})=n-2$ and larger than all other elements and let $M_w$ be the total order obtained from $M_{w'}$ by letting $n-2$ larger than all elements of $M_{w'}$. Clearly, $(L_w, M_w)$ is a realizer of  $P_{w}$ (indeed,  $n-2$ and $n-1$ are incomparable in $L_w\cap M_w$ and for all $x\in \{-1,0,\ldots,n-3\}$, $x< n-1$ in $L'_w \cap M'_w$ proving that $\{L_w,M_w\}$ is a realizer of $P_w$). By construction $n-1$ is maximal in $P_w$ and $M^{'}_w$. The proof of the lemma  is now complete.
\end{proof}

It is easy to prove and well known that there are $2^{\aleph_0}$ hereditary classes  of  finite permutation graphs. This is due to the existence of infinite antichains among finite permutation graphs.

In general, it is not true that two words with different sets of finite factors give different  ages. Let us recall two basic notions of Symbolic Dynamic. A  word $u$  is \emph{recurrent} if every finite factor occurs infinitely often; the word
 $u$   is \emph{uniformly recurrent} if for every $n\in \NN$ there exists $m\in \NN$ such that each factor $u(p)\ldots u(p+n)$ of length $n$ occurs as a factor of every factor of length $m$.

\begin{theorem}\label{thm:recurrent-word} Let $\mu$ and $\mu'$ be two words. If $\mu$ is recurrent and   $\age (G_\mu) \subseteq \age(G_{\mu'})$, then $Fac(\mu)\subseteq Fac(\mu')$.
\end{theorem}

%
%

Using this result and the fact that  there are $2^{\aleph_0}$ $0$-$1$ recurrent words with distinct sets of factors, we obtain the following.

\begin{theorem}\label{thm:cont-ages} There are $2^{\aleph_0}$ ages of permutation graphs.
\end{theorem}

The ages we obtain in Theorem \ref{thm:cont-ages} are not necessarily well-quasi-ordered. To obtain well-quasi-ordered ages,  we consider graphs associated to uniformly recurrent sequences.

\begin{theorem}\label{thm:uniformly-ages}Let $\mu$ be a $0$-$1$ sequence on an infinite interval of $\ZZ$. The following propositions are equivalent.
\begin{enumerate}[$(i)$]
  \item $\mu$ is uniformly recurrent.
  \item $\mu$ is recurrent and $\age(G_\mu)$ is minimal prime.
\end{enumerate}
\end{theorem}


As it is well known,  there are $2^{\aleph_0}$ uniformly recurrent words with distinct sets of factors (e.g. Sturmian words with different slopes, see Chapter 6 of  \cite{pytheas}). With Theorem  \ref{thm:recurrent-word} we get:

\begin{theorem}\label{thm:minprimeages} There are $2^{\aleph_0}$ ages of permutation graphs which are minimal prime.
\end{theorem}

Permutation graphs come from posets and  from bichains. Let us recall that a \emph{bichain} is relational structure $R:= (V, (\leq', \leq''))$ made of a set $V$ and two linear orders $\leq'$ and $\leq''$ on $V$. If $V$ is finite and has $n$ elements, there is a unique permutation $\sigma$ of $\{1, \ldots, n\}$ for which $R$ is isomorphic to the  bichain $C_{\sigma}:= (\{1, \ldots, n\}, \leq, \leq_{\sigma})$ where $\leq$ is the natural order on $\underline n:=\{1, \ldots, n\}$ and  $\leq_{\sigma}$ is the linear order defined by $i\leq_{\sigma} j$ if $\sigma(i)\leq \sigma (j)$.


If we represent bichains by permutations, embeddings between bichains is equivalent to the \emph{ pattern   containment}  between the corresponding permutations, see Cameron \cite{cameron}.

To a bichain $R:= (V, (\leq', \leq''))$, we may associate the intersection order $o(R):= (V, \leq'\cap\leq'')$ and to $o(R)$ its comparability graph.

The following is Theorem 67 from \cite{pouzet-zaguia-w.q.o. 22}.

\begin{theorem}\label{thm:minimalprime-graph-poset}
\begin{enumerate}[$(1)$]
\item Let $P:=(V,\leq)$ be a  poset.  Then  $\age(\inc (P))$ is minimal prime if and only if
$\age(\comp(P))$ is minimal prime. Furthermore,  $\age(P)$ is minimal prime if and only if $\age (\ainc(P))$ is minimal prime and $\downarrow  \prim (\age(P))= \age (P)$.
\item  Let $B:=(V,(\leq_1, \leq_2))$ be a bichain and $o(B):= (V, \leq_1\cap \leq_2)$. Then  $\age (B)$ is minimal prime if and only if $\age (o(B))$ is minimal prime and $\downarrow  \prim (\age (B))= \age (B)$.
\end{enumerate}
\end{theorem}

%

%

%

With $(ii)$ of Theorem \ref{thm:minimalprime-graph-poset} and Theorem \ref {thm:minprimeages} we  have:

\begin{theorem}\label{thm:minprimeages-bichains}
There are $2^{\aleph_0}$ ages of bichains and permutation orders  which are minimal prime.
\end{theorem}


We now are able to state our complete characterization of minimal prime graphs. The corresponding characterization of minimal prime posets and bichains will follow from Theorem \ref{thm:minimalprime-graph-poset} and a careful examination of our list of graphs to decide which graphs are comparability graphs.

\begin{figure}[h]
\begin{center}
\leavevmode \epsfxsize=3in \epsfbox{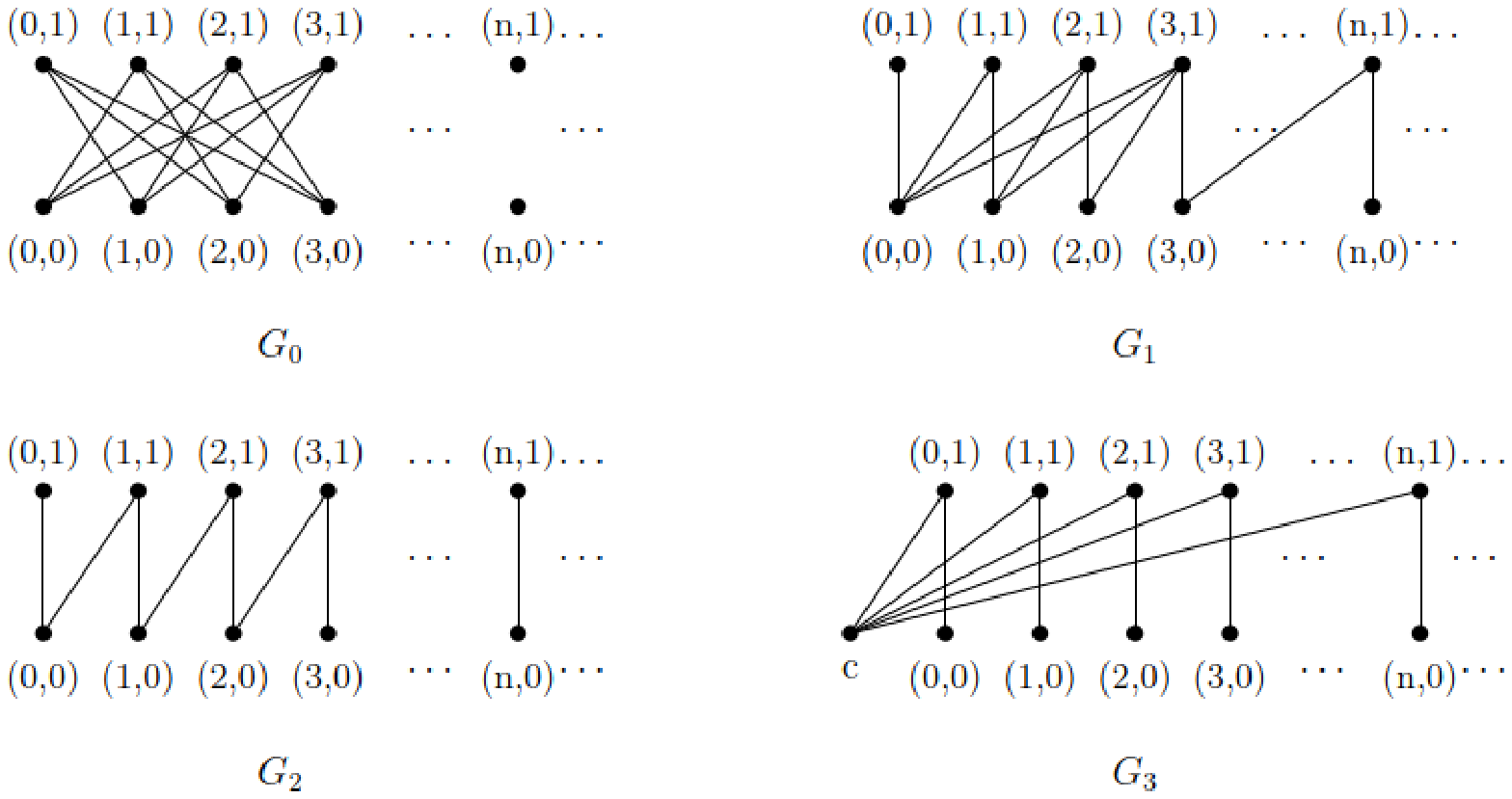}
\end{center}
\caption{} \label{fig:list-graph-min-a}
\end{figure}

\begin{figure}[h]
\begin{center}
\leavevmode \epsfxsize=3in \epsfbox{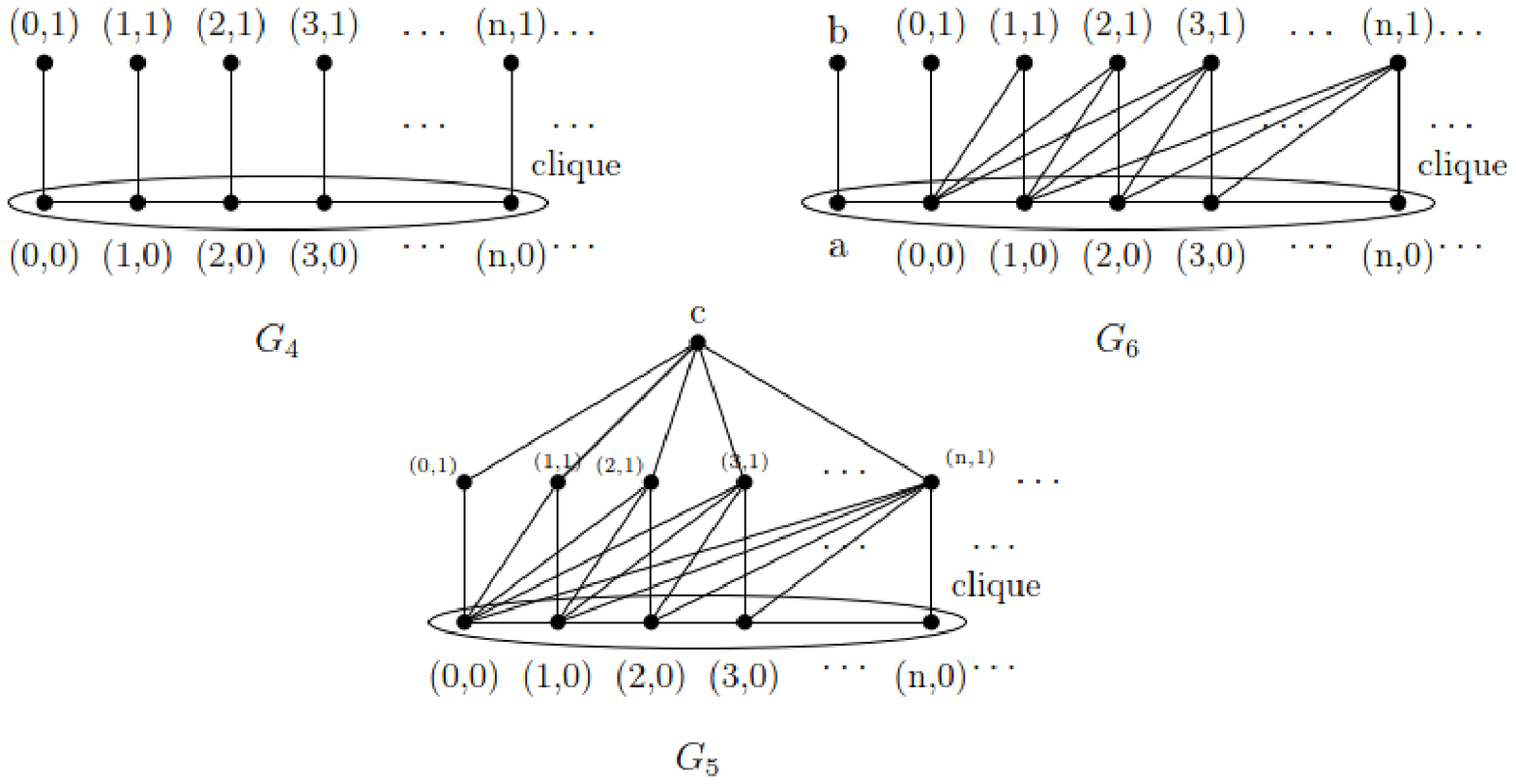}
\end{center}
\caption{} \label{fig:list-graph-min-b}
\end{figure}
\begin{figure}[h]
\begin{center}
\leavevmode \epsfxsize=3in \epsfbox{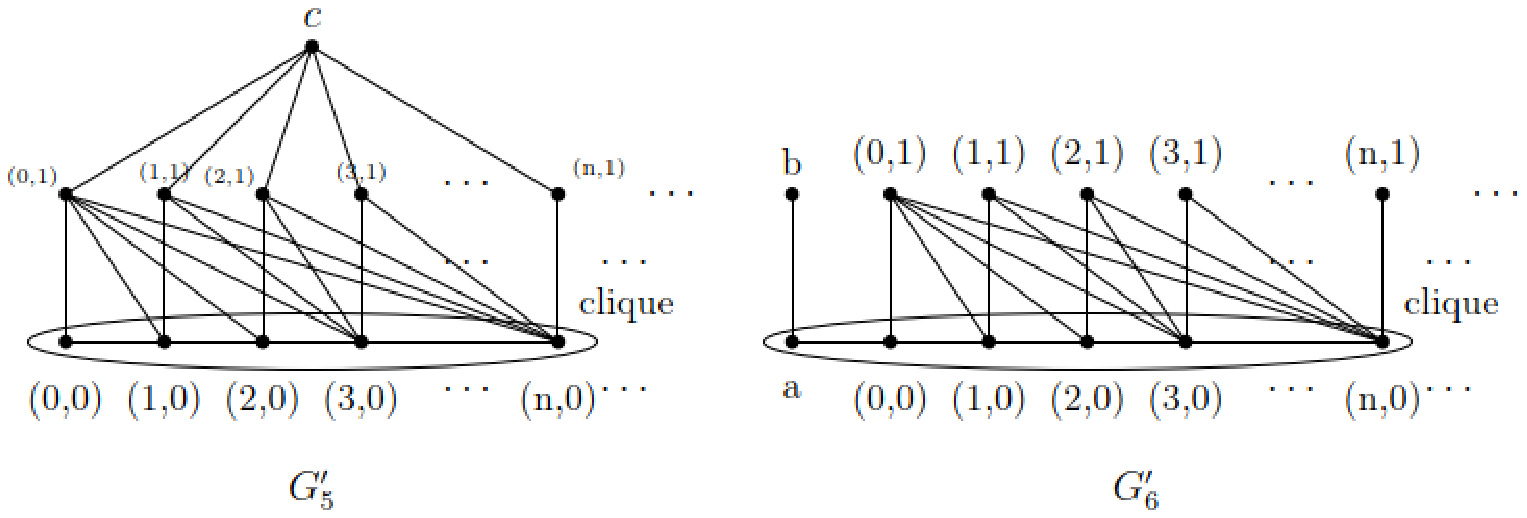}
\end{center}
\caption{} \label{fig:list-graph-min-c}
\end{figure}

These graphs were considered first by the last two authors  in \cite{pouzet-zaguia2009}. It was shown in \cite{pouzet-zaguia2009} that every prime graph with no infinite clique embeds one of the graphs depicted in Figure \ref{fig:list-graph-min-a}.

Let $\mathcal L:= \{\age(G_0), \age(\overline G_0), \age (G_1), \age(\overline G_1), \age(G_3), \age(\overline G_3), \age (G_4), \age(\overline G_4),\break \age(G_5), \age (G_6), \age(\overline G_6)\}$.

The corresponding  graphs are depicted in Figures \ref {fig:list-graph-min-a}, \ref{fig:list-graph-min-b} and \ref {fig:list-graph-min-c}. It should be noted that the graphs $G_5$, $\overline G_5$ and $G'_5$ have the same age. Also, $G_6$,  $G'_6$ have the same age.

\begin{theorem}\label{thm:charact-minimal-prime-ages}
A hereditary class $\mathcal{C}$ of finite graphs is minimal prime if and only if $\mathcal{C}=\age(G_\mu)$ for some uniformly recurrent word on $\NN$, or $\mathcal{C}\in \mathcal {L}$.
\end{theorem}

\begin{proof}
$\Leftarrow$. Follows from Theorem \ref {thm:uniformly-ages} and Chapter 6 page 109 of the first author's thesis \cite{oudrar}.

 $\Rightarrow$ Follows essentially from Theorem \ref{thm:chudnovsky}.  Let $\mathcal C$ be a minimal prime age. Then $\mathcal C$ contains infinitely many prime graphs  of one of the types  given in Theorem \ref{thm:chudnovsky}.   If for an example, $\mathcal C$ contains infinitely many chains,  that is graphs of the form $G_{\mu}$ for $\mu$  finite, then, since it is minimal prime, this is the age of some $G_{\mu}$ with $\mu$ uniformly recurrent. For the other cases, use the structure of the infinite graphs described in  Figures \ref{fig:list-graph-min-a}, \ref {fig:list-graph-min-b} and \ref{fig:list-graph-min-c}.
\end{proof}

\begin{theorem}
\begin{enumerate}[$(1)$]
  \item A minimal prime hereditary class $\mathcal{C}$ of finite graphs is hereditary  well-quasi-ordered if and only if $\mathcal{C}\in \mathcal L$.
  \item A minimal prime hereditary class $\mathcal{C}$ of finite graphs remains well-quasi-ordered when just one  label is added if and only if $\mathcal{C}=\age(G_\mu)$ for some periodic $0$-$1$ word on $\NN$, or $\mathcal{C}\in \mathcal L$.
  \end{enumerate}
\end{theorem}

\begin{corollary}
\begin{enumerate}[$(1)$]
  \item A hereditary class $\mathcal{C}$ of finite comparability graphs is minimal prime if and only if $\mathcal{C}=\age(G_\mu)$ for some uniformly recurrent word on $\NN$, or \\
   $\mathcal{C}\in  \{\age(G_0), \age (G_1), \age(\overline G_1), \age(G_3), \age( G_5), \age(G_6), \age(\overline G_6)\}$.

  \item A hereditary class $\mathcal{C}$ of finite permutation graphs is minimal prime if and only if $\mathcal{C}=\age(G_\mu)$ for some uniformly recurrent word on $\NN$, or \\
   $\mathcal{C}\in  \{ \age (G_1), \age(\overline G_1), \age(G_5), \age(G_6), \age(\overline G_6)\}$.
\end{enumerate}
\end{corollary}

\section{Bounds of minimal prime hereditary classes}

We recall that a \emph{bound} of a hereditary class $\mathcal C$ of finite structures (e.g. graphs, ordered sets) is any structure $\mathcal R\not \in \mathcal C$ such that every proper induced substructure of $\mathcal R$ belongs to $\mathcal C$.

\begin{theorem}\label{thm:bound-uniform}Let $\mu$ be a uniformly recurrent  and non-periodic $0$-$1$ word. Then $\age(G_\mu)$ has infinitely many bounds.
\end{theorem}

If $\mu$ is periodic, $\age(G_\mu)$ may have infinitely many bounds. This is the case if $\mu$ is constant. For nonconstant $0$-$1$ word, we propose the following conjecture.\\

\noindent \textbf{Conjecture:} If $\mu$ is a periodic nonconstant $0$-$1$ word, then $G_{\mu}$ has finitely many bounds.\\

As for Theorem  \ref{thm:bound-uniform}, we have.

 \begin{theorem}\label{thm:bound-uniformages}Let $\mu$ be a uniformly recurrent  and non-periodic $0$-$1$ word; let $B_{\mu}$ be a bichain such that the comparability graph of the intersection order is $G_{\mu}$ and $P_{\mu}$  be a transitive orientation of $G_{\mu}$. Then $\age(B_\mu)$ and $\age(P_\mu)$ have infinitely many bounds.
\end{theorem}

\end{document}